\documentclass[reqno]{amsart}
\usepackage{amssymb,amscd,verbatim, amsthm,graphicx, color}
\usepackage[enableskew,vcentermath]{youngtab}
\usepackage{srcltx}
\usepackage[mathscr]{eucal}

\newcommand \fk[1]{{{\mathfrak #1}}}
\newcommand \C[1]{{\mathcal #1}}
\newcommand \ovl[1]{{\overline {#1}}}

\newcommand \bb[1]{{\mathbb #1}}

\newcommand\fg{\mathfrak g}

\newcommand \bA{{\mathbb A}}
\newcommand \bC{{\mathbb C}}

\newcommand \bH{{\mathbb H}}
\newcommand \bR{{\mathbb R}}
\newcommand \bZ{{\mathbb Z}}

\newcommand\ue{{\underline e}}
\newcommand\uh{{\underline h}}

\newcommand\cha{{\check \alpha}}

\newcommand\CB{{\C B}}

\newcommand\ie{{\it i.e.~}}

\newcommand\eg{{\it e.g.~ }}
\newcommand\etc{{\it etc.~ }}

\newcommand\ep{{\epsilon}}
\newcommand\la{{\lambda}}
\newcommand\om{{\omega}}
\newcommand\al{{\alpha}}
\newcommand\sig{{\sigma}}

\newcommand\fh{{\mathfrak h}}

\newcommand\sgn{\mathsf{sgn}}
\newcommand\gen{\mathsf{gen}}
\newcommand\Wt{\mathsf{Wt}}

\newtheorem{proposition}{Proposition}[subsection]
\newtheorem{corollary}[proposition]{Corollary}
\newtheorem{lemma}[proposition]{Lemma}
\newtheorem{theorem}[proposition]{Theorem}

\theoremstyle{definition}
\newtheorem{definition}[proposition]{Definition}
\newtheorem{remark}[proposition]{Remark}

\newtheorem{example}[proposition]{Example}

\newcommand{\clrblu}{\color{blue}}

\newcommand{\clrr}{\color{red}}

\newcommand\End{\operatorname{End}}

\newcommand\Id{\operatorname{Id}}
\newcommand\Ad{\operatorname{Ad}}

\newcommand\ad{\operatorname{ad}}

\newcommand\cc{\textsf{cc}}

\newcommand\St{{\mathsf{St}}}

\numberwithin{equation}{subsection}

\newfont{\lge}{cmmi10 scaled 1640}

\begin{document}



\bigskip
\title{Ladder representations of $GL(n,\mathbb Q_p)$} 
\author{Dan Barbasch}
       \address[D. Barbasch]{Dept. of Mathematics\\
               Cornell University\\Ithaca, NY 14850}
       \email{barbasch@math.cornell.edu}

\author{Dan Ciubotaru}
        \address[D. Ciubotaru]{Mathematical Institute, University of Oxford, Oxford, OX2 6GG}
        \email{dan.ciubotaru@maths.ox.ac.uk}

\dedicatory{To David with admiration}

\begin{abstract}
In this paper, we recover certain known results about the ladder
representations of $GL(n,\mathbb Q_p)$ defined and studied by Lapid, M\'
inguez, and Tadi\'c.  We work in the equivalent setting of graded Hecke algebra
modules. Using the Arakawa-Suzuki functor from category $O$ to graded
Hecke algebra modules, we show that the determinantal formula proved
by Lapid-M\'inguez and Tadi\'c  is a direct
consequence of the BGG resolution of finite dimensional simple
$gl(n)-$modules. We make a connection between the semisimplicity of
Hecke algebra modules, unitarity with respect to a certain hermitian
form, and ladder representations. 
\end{abstract}

\maketitle

\setcounter{tocdepth}{1}
\tableofcontents

\section{Introduction} 
{In this paper we study  a class of representations of the graded affine Hecke algebra which are unitary for a star operation which we call $\bullet$. The $\bullet$-unitary dual for type $A$ is determined completely. In this case, the unitary dual corresponds via the Borel-Casselmann equivalence of categories \cite{Bo} composed with the reduction to the affine graded Hecke algebra of \cite{L1} to the \textit{ladder representations} defined and studied in \cite{LM} and \cite{Ta} for 
$GL(n,\bb Q_p)$.}

\medskip
The classification of the unitary dual of real and $p$-adic reductive groups is
one of the central problems of representation theory. Typically, by
results of Harish-Chandra, this
problem is reduced to an algebraic one, the study of 
\textit{admissible representations}
of an algebra endowed with a \textit{star operation}. In the case of
real groups, this algebra is the enveloping algebra, in the case of
$p$-adic groups, an Iwahori-Hecke type algebra with parameters. In both
cases, the  star operation is derived from the antiautomorphism 
$g\mapsto g^{-1}$. In the real case, David Vogan and his collaborators
\cite{Vetal} make a deep study of signatures of hermitian forms of
admissible modules by exploiting the relationship between two
different star operations, one related to the real form of the
reductive group, the other related to the compact form of the group. 
Motivated by this, we study the analogues of these star
operations for the graded affine Hecke algebras. The star
operation coming from the $p$-adic group is made explicit in
\cite{BM2}. In \cite{BC-sanya}, we introduce and study another star
operation which we denote by $\bullet$, the analogue of the star
operation for a compact form. The problem of the unitarity
of representations for $\bullet$ seemed an artificial one. However,
the results of Opdam \cite{Op}, {and more recently Oda \cite{Od}}, show that \textit{spherical representations}
of graded affine Iwahori-Hecke algebras play an important role in
harmonic analysis of symmetric spaces of compact type. Motivated by
this result, we initiated a systematic study of $\bullet-$unitary
representations. This is the topic of this paper. 

\bigskip
The first set of results is a connection between $\bullet-$unitary
representations, and representations which are $\bb
A-$semisimple. This is the content of  Propositions \ref{p:unit-ss}
and \ref{p:ssconverse}. This provides a connection to the work of
\cite{Ch},\cite{KR},  and \cite{Ra}. 

\medskip
In ongoing research we are planning to
determine the entire $\bullet-$unitary dual for graded affine Hecke algebras of arbitrary
type. The most complete results to date are for type A. In the
process we found the links to the {ladder representations} in
the title, and the results \cite{LM}, \cite{CR}, and \cite{Ta}. 

\medskip
A seminal idea, pioneered by D. Vogan, was to try to make a connection between
the unitary dual of real and $p$-adic groups via intertwining
operators, via \textit{petite} $K-$types and $W-$types. 
This was developed systematically by the authors of this
paper, jointly and separately, in particular to determine the full
spherical unitary dual of split $p$-adic (and split real classical) groups.
We follow this approach in this paper. We relate the $\bullet-$unitary
(star for the compact form of the Lie algebra) dual of Verma modules
to the $\bullet-$unitary dual of the graded affine Hecke algebra using the
functors introduced by Arakawa and Suzuki, \cite{AS,Su}.  The advantage of this
method is that it provides interesting connections between the Bernstein-Gelfand-Gelfand
resolution and results about character formulas of  {ladder
  representations}.     

\medskip
Some time ago, motivated by conjectures of Arthur concerning unipotent
representations, D. Barbasch, S. Evens, and A. Moy
conjectured the existence of an action of the Iwahori-Hecke
algebra  on the cohomology of the incidence variety of a pair of
nilpotent elements whose $sl(2)-$triples commute (the conjectures
actually referred to more general pairs). In section 5 we provide
evidence for this conjecture, by establishing connections to the
work of \cite{Gi} and \cite{EP}.  

\medskip

\noindent{\bf Acknowledgements.}  The first author was partially supported by NSF grants
  DMS-0967386, DMS-0901104 and an NSA-AMS grant. The second author was
  partially supported by NSF DMS-1302122 and NSA-AMS 111016. The
  authors thank Eitan Sayag for interesting discussions about ladder
  representations.

\section{The $\bullet$ star operation}\label{sec:1}

\subsection{Graded affine Hecke algebra}

Let $\Phi=(V,R,V^\vee, R^\vee,\Pi)$ be a  reduced based $\bR$-root system . Let
$W\subset GL(V)$ be the Weyl group generated by the simple reflections
$\{s_\al: \al\in\Pi\}.$
The positive roots are $R^+$ and the positive coroots are $R^{\vee,+}.$ The complexifications of $V$ and $V^\vee$ are denoted by
$V_\bC$ and $V^\vee_\bC$, respectively, and we denote by $\bar{\ }$
the complex conjugations of $V_\bC$ and $V^\vee_\bC$ induced by $V$
and $V^\vee$, respectively. 

Let $k: \Pi\to \bR_{>0}$ be a function such that $k_\al=k_{\al'}$ whenever
$\al,\al'\in \Pi$ are $W$-conjugate. Let $\bC[W]$ denote the group
algebra of $W$ and $S(V_\bC)$ the symmetric algebra over $V_\bC.$ The
group $W$ acts on $S(V_\bC)$ by extending the action on $V.$ For every
$\al\in \Pi,$  denote the difference operator by
\begin{equation}\label{e:diffop}
\Delta: S(V_\bC)\to S(V_\bC),\quad
\Delta_\al(a)=\frac{a-s_\al(a)}{\al},\text{ for all }a\in S(V_\bC).
\end{equation}
{If $a\in V_\bC$, then $\Delta_\al(a)=\langle a,\al^\vee\rangle$.}
\begin{definition}[\cite{L1}]\label{d:graded}
The graded affine Hecke algebra $\bH=\bH(\Phi,k)$  is the unique
associative unital algebra generated by $\{a: a\in S(V_\bC)\}$ and
$\{t_w: w\in W\}$ such that  
\begin{enumerate}
\item[(i)] the assignment $t_wa\mapsto w\otimes a$ gives an
isomorphism  of $(\bC[W],S(V_\bC))-$bimodules between 
$\bH$ and $\bC[W]\otimes S(V_\bC)$; 
\item[(ii)] $a t_{s_\al}=t_{s_\al}s_\al(a)+k_\al \Delta_\al(a),$
  for all $\al\in \Pi$, $a\in S(V_\bC).$
\end{enumerate}
\end{definition}

{The center of $\bH$ is $S(V_\bC)^W$ (\cite{L1}). By Schur's Lemma, the center of $\bH$
acts by scalars on each irreducible $\bH$-module. The central
characters are parameterized by $W$-orbits in $V_\bC^\vee.$ If $X$ is
an irreducible $\bH$-module, denote by $\cc(X)\in V_\bC^\vee$
(actually in $ W\backslash
V_\bC^\vee$) its central character. 
}

\subsection{Star operations}
Let $w_0$ denote the long Weyl group element, and let $\delta$ be the
involutive automorphism of $\bH$ determined by
\begin{equation}\label{autom}
\delta(t_w)=t_{w_0 w w_0},\ w\in W,\quad \delta(\omega)=-w_0(\omega),\
\omega\in V_\bC.
\end{equation}
When $w_0$ is central in $W$, $\delta=\Id$.

\begin{definition}
Let $\kappa:\bH\to \bH$ be a conjugate linear involutive algebra
anti-automorphism. An $\bH$-module $(\pi,X)$ is said to be
$\kappa$-hermitian if $X$ has a hermitian form $(~,~)$ which is
$\kappa$-invariant, \ie, 
$$
(\pi(h)x,y)=(x,\pi(\kappa(h))y),\quad x,y\in X,\  h\in\bH.
$$
A hermitian module $X$ is $\kappa$-unitary if the $\kappa$-hermitian
form is positive definite. 
\end{definition}

\begin{definition}\label{d:basicstars}
Define a conjugate linear algebra anti-involution $\star$ of $\bH$ by
\begin{align}\label{e:star}
t_{w}^\star=t_{w^{-1}},\ w\in W,\quad \omega^\star={\ovl{\Ad
  t_{w_0}(\delta (a))}}=-t_{w_0}\cdot
\overline{w_0(\omega)}\cdot t_{w_0},\ \omega\in V_\bC. 
\end{align}
Similarly define $\bullet$ by
\begin{align}\label{e:bullet}
t_{w}^\bullet=t_{w^{-1}},\ w\in W,\quad
\omega^\bullet=\overline\omega,\ \omega\in V_\bC. 
\end{align}
\end{definition}
The operations $\star$ and $\bullet$ are related by
\begin{equation}\label{e:relstar}
 {\star=\Ad t_{w_0}\circ\bullet\circ\delta},
\quad \text{ for all } h\in\bH.
\end{equation}

\begin{remark}
In \cite{BC-sanya}, it is proved that $\star$ and $\bullet$ are the
only star operations of $\bH$ that satisfy certain natural
conditions. 
{When $\bH$ is obtained by grading
the Iwahori-Hecke algebra of a reductive $p$-adic group,
$\star$ corresponds to the natural star 
operation of the $p$-adic group}. The
operation $\bullet$ is the analogue of the compact star operation
defined for real reductive groups in  \cite{Vetal}. 
\end{remark}

\subsection{Semisimplicity}
In this section, suppose the parameters $k_\al$ are positive, but
arbitrary.
Let $(\pi,X)$ be a finite dimensional $\bH$-module. For every $\la\in
V_\bC^*$, define
\begin{equation}
\begin{aligned}
X_\la&=\{ x\in X: \pi(\om)x=(\om,\la)x,\ \text{for all }\om\in
V_\bC\},\\
 X_\la^\gen&=\{ x\in X: (\pi(\om)-(\om,\la))^nx=0 \text{ for some
 }n\in\mathbb N, \text{ for all }\om\in
V_\bC\}.\\
\end{aligned}
\end{equation}
A functional $\la\in V_\bC^*$ is called a weight of $X$ if $X_\la\neq
0.$ Let $\Wt(X)$ denote the set of weights of $X$. It is straightforward that $\Wt(X)\subset
W\cdot \cc(X).$
\begin{definition}\label{def:ss}
The module $(\pi,X)$ is called $\bA$-semisimple if 
$X_\la=X_\la^\gen$ for all $\la.$
\end{definition}


\begin{proposition}\label{p:unit-ss}
Assume $(\pi,X)$ is a $\bullet$-unitary $\bH$-module. Then $X$
is $\bA$-semisimple.
\end{proposition}

\begin{proof}
Let $(~,~)_X$ be the positive definite $\bullet$-form on $X$. Let $\la$ be a
weight of $X$ and $x_\la\neq 0$ a weight
vector. Define $$\{x_\la\}^\perp=\{y\in X: (x_\la,y)_X=0\}.$$
Let $y\in \{x_\la\}^\perp$ be
given. Since $$0=(\al,\la)(x,y)_X=(\pi(\om)x_\la,y)_X=(x_\la,\pi(\om^\bullet)y)_X=(x_\la,\pi(\overline\om)y),\quad
\om\in V_\bC,$$
it follows that $\{x_\la\}^\perp$ is $\bA$-invariant. Since the form
$(~,~)_X$ is positive definite, we have $X=\bC x_\la\oplus
\{x_\la\}^\perp$ as $\bA$-modules. By induction, it follows that $X$
is a direct sum of one-dimensional $\bA$-modules, thus $\bA$-semisimple.
\end{proof}

\begin{remark}
The above proposition can also be interpreted as the following linear
algebra statement:  if $J$ is a hermitian matrix, and $N$ a nonzero
  nilpotent matrix such that
\[
JN=N^*J,
\]
then $J$ is not positive definite.
\end{remark}

\begin{remark}
The proof and statement of Proposition
\ref{p:unit-ss} can be easily generalized by replacing $\bA$ with any
parabolic subalgebra of $\bH.$
\end{remark}
{
\begin{proposition}[\cite{BC}]
  \label{p:ssconverse}
Assume the central character $\chi$  of $\pi$ satisfies
$$|\langle\chi,\cha \rangle|\ge 1,$$ for all simple roots $\al.$ If $\pi$ is
$\bA-$semisimple, it is $\bullet-$unitary.
\end{proposition}
This is a converse to Proposition \ref{p:unit-ss} and more difficult. 
We refer to \cite{BC} for a proof. Notice that, in particular,
Propositions \ref{p:unit-ss} and \ref{p:ssconverse} imply that at
integral central character $\chi$, \ie
$\langle\chi,\al^\vee\rangle\in\bZ$ for all roots $\al$, a simple
$\bH$-module is $\bullet$-unitary if and only if it is
$\bA$-semisimple. 
}
\section{Ladder representations: definitions}\label{sec:2}

We consider the graded Hecke algebra of type $A$. In this case, we can
classify the $\bullet$-unitary dual. We begin by recalling
Zelevinsky's classification \cite{Z} of the simple modules. We will
phrase the classification ``with quotients'' rather than
``submodules'', cf. \cite[\S10]{Z}. 

\subsection{Multisegments} We restrict to  $\bH$-modules with real
central character. By \cite[Corollary 4.3.2 or Corollary 5.1.3]{BC},
every simple $\bH$-module with real central character admits a
nondegenerate $\bullet$-invariant hermitian form. 

A segment is a set $\Delta=\{a,a+1,a+2,...,b\}$, where $a,b\in\bR$ and
$a\equiv b$ (mod $\bZ$). We will write $\Delta=[a,b]$ and $|\Delta|=b-a+1$ for the length. A multisegment
is an ordered collection $(\Delta_1,\Delta_2,\dots,\Delta_r)$ of
segments. Following \cite[\S4.1]{Z}, two segments $\Delta_1$ and
$\Delta_2$ are called 
\begin{enumerate}
\item[(a)] linked, if $\Delta_1\not\subset\Delta_2$,
  $\Delta_2\not\subset\Delta_1$, and $\Delta_1\cup \Delta_2$ is a
  segment; 
\item[(b)] juxtaposed, if $\Delta_1,\Delta_2$ are linked and
  $\Delta_1\cap\Delta_2=\emptyset$;  
\end{enumerate}
One says that 
\begin{enumerate}
\item[(c)] $\Delta_1$ precedes $\Delta_2$ if $\Delta_1,\Delta_2$ are
  linked and $a_1<a_2.$ 
\end{enumerate}

For every segment $\Delta$ with $m=b-a+1$, let $\langle\Delta\rangle$
denote the one-dimensional $\bH_m$-module which extends the sign
$W$-representation and on which $\bA$ acts by the character
$\bC_{[a,b]}.$ If $(\Delta_1,\Delta_2,\dots,\Delta_r)$ is a
multisegment, denote by  
\begin{equation}
\langle\Delta_1\rangle\times\langle\Delta_2\rangle\times\dots\times\langle\Delta_r\rangle 
\end{equation}
the induced module $\bH_n\otimes_{\bH_{m_1}\times \bH_{m_2}\times\dots\times\bH_{m_r}} (\langle\Delta_1\rangle\boxtimes\langle\Delta_2\rangle\boxtimes\dots\boxtimes\langle\Delta_r\rangle)$, where $m_i=b_i-a_i+1$ and $n=\sum m_i.$

We need two of the main results from \cite{Z}.

\begin{theorem}[{\cite[Theorem 4.2]{Z}}]\label{t:irred-ind}
The following conditions are equivalent:
\begin{enumerate}
\item The module $\langle\Delta_1\rangle\times\dots\times\langle\Delta_r\rangle$ is irreducible.
\item For each $i,j=1,\dots,r$, the segments $\Delta_i$ and $\Delta_j$ are not linked.
\end{enumerate}
\end{theorem}

\begin{theorem}[{\cite[Theorem 6.1]{Z}}]\label{t:zel-class}
\begin{enumerate}
\item[(a)] Let $(\Delta_1,\dots,\Delta_r)$ be a multisegment. Suppose
  that for each $i<j$, $\Delta_i$ does not precede $\Delta_j$. Then
  the representation
  $\langle\Delta_1\rangle\times\dots\times\langle\Delta_r\rangle$ has
  a unique irreducible quotient denoted by $\langle
  \Delta_1,\dots,\Delta_r\rangle.$ 
\item[(b)] The modules $\langle \Delta_1,\dots,\Delta_r\rangle$ and
  $\langle \Delta_1',\dots,\Delta_s'\rangle$ are isomorphic if and
  only if the corresponding multisegments are equal up to a
  rearrangement. 
\item[(c)] Every simple $\bH_n$-module is isomorphic to one of the
  form
  $\langle\Delta_1\rangle\times\dots\times\langle\Delta_r\rangle$. 
\end{enumerate}
\end{theorem}
{\begin{remark}
For the most part, the above results   are instances of the Langlands
classification. A multisegment corresponds to data $(M,\sig,\nu)$
where $$M=GL(b_1-a_1+1)\times\dots\times GL(b_r-a_r+1)$$
is a Levi component,
the tempered representation $\sig$ is the Steinberg, and the $(a_i,b_i)$
determine the $\nu.$ The fact that $\Delta_i$ precedes $\Delta_j$ is
the usual dominance condition for $\nu.$ The remaining results are
sharpenings of the Langlands classification in the case of $GL(n).$ 
\end{remark}
}
\begin{definition}[Ladder representations \cite{LM}]\label{d:ladder} Let
  $\Delta_i=[a_i,b_i]$ $1\le i\le r$ be Zelevinsky segments. If
  $a_1>a_2>\dots>a_r$ and $b_1>b_2>\dots>b_r$, call the irreducible
  representation $\langle\Delta_1,\Delta_2,\dots,\Delta_r\rangle$ a
  ladder representation.
\end{definition}

\begin{example}[Speh representations \cite{BM}]\label{ex:speh} Let
  $\Delta_i$, $1\le 
  i\le r$ be segments as in Definition \ref{d:ladder}, such that
  $b_i-a_i+1=d$ for a fixed $d$ and $a_{i}-a_{i+1}=1$ for all
  $i$. Then $\langle\Delta_1,\dots,\Delta_r\rangle$ is
  irreducible as an $S_n$-representation, isomorphic to the
  $S_n$-representation parameterized by the rectangular Young diagram
  with $r$ rows and $d$ columns. These modules are both
  $\bullet$-unitary and $\star$-unitary (\cite{BM,CM}) and correspond to the
  ($I$-fixed vectors) of Speh representations.
\end{example}

\subsection{Cherednik's construction}\label{s:cherednik}
\ 
{
As in Definition \ref{d:ladder}, let $\langle\Delta_1,\dots,\Delta_r\rangle$, $a_1>a_2>\dots>a_r$, $b_1>b_2>\dots>b_r$ be a ladder representation.
The interesting case  is
when $\Delta_i$ is linked to $\Delta_{i+1}$ for all $i$. In fact,
since 
tensoring with a character of the center of $\bH$ does not change
$\bA$-semisimplicity, we may even assume that $a_i,b_i\in\bZ$ for all
$i$. From now on, this type of ladder representations will be called
{\it integral}.

\smallskip

Following \cite{Ch}, we give a
combinatorial construction of integral 
ladder representations.  Let $\langle\Delta_1,\dots,\Delta_r\rangle$ be an integral ladder representation with the notation as above. 
Set 
\begin{equation}\label{eq:segments}
\lambda=(a_1,\dots, b_1, a_2,\dots,b_2,\dots,a_r,\dots,b_r)\in\bZ^n
\end{equation}
viewed as an element of $V_\bC^\vee\cong \bC^n$. The underlying
multisegment $(\Delta_1,\dots,\Delta_r)$ gives a 
skew-Young diagram, where each box in the Young diagram corresponds to
an integer in one of the multisegments. 
{More precisely, the underlying skew diagram is formed as follows. The
  first segment $\Delta_1$ gives the top row with $|\Delta_1|$ boxes,
  each box for one of the integers in $\Delta_1$ in order. The segment
  segment $\Delta_2$ gives the second row with $|\Delta_2|$ boxes,
  immediately below the first, etc. The rows are aligned so that
\begin{enumerate}
\item the two boxes  are in the same column if and only if the they correspond to the same integer in the multisegment, and
\item two boxes is two adjacent columns correspond to two consecutive integers in the multisegment.
\end{enumerate}
}
For example, if the multisegment is $([2,4],[0,2],[-2,-1])$, the resulting 
 skew Young diagram is 
\begin{equation}
\young(::::\hfill\hfill\hfill,::\hfill\hfill\hfill,\hfill\hfill).
\end{equation}
Notice that the skew Young diagram does not recover the integral
ladder representation uniquely, only up to tensoring with a central
character. However, if we specify an integer $a$ such that the first
segment starts with $a$, then the multisegment is determined.  
}

\smallskip
\smallskip

{We fix a skew Young diagram as above and we will form skew Young tableaux with that shape.} Let $[1\dots n]$ be the set of
integers $1,2,\dots,n$. Let $Y_1$ be the skew Young
tableau with entries in $[1\dots n]$ such that in the first row, the entries
are, in order, $1, 2,\dots, b_1-a_1+1$, in the second row,
$b_1-a_1+2,b_1-a_1+3,\dots, b_1+b_2-a_1-a_2+2$, etc. {In our example,
\begin{equation}
Y_1=\young(::::123,::456,78).
\end{equation}
}
Consider all skew Young tableaux with entries in $[1\dots n]$ subject to the
requirements:
\begin{enumerate}
\item the entries  are increasing left-right on each row;
\item the entries are increasing up-down on each $45^\circ$-diagonal.
\end{enumerate}

Denote every such tableau by $Y_w$, where $w\in S_n$ is the
permutation transforming $(1,2,\dots,n)$ to the entries of the tableau
read in order from the top row to the bottom row and on each row from
left to right. Let
\begin{equation}\label{e:W-sigma}
W(\Delta_1,\dots,\Delta_r)=\text{ the set of }w\in S_n\text{ parameterizing the tableaux
}Y_w\text{ for }(\Delta_1,\dots,\Delta_r).
\end{equation}

\begin{theorem}[Cherednik {\cite[Theorem 4]{Ch}}, see also Ram {\cite{Ra}}]\label{t:Ch} The set $\{Y_w\}$
  defined above is a basis of a simple $\bH$-module $\mathsf
  C(\Delta_1,\dots,\Delta_r)$ such that 
\begin{enumerate}
\item $Y_w$ is an $\bA$-weight vector with weight $w(\lambda).$ 
\item the action of $W$ on $\{Y_w\}$ is as follows:
\begin{equation}
\pi(t_{s_\al})Y_w=\begin{cases} \frac {1}{(\al,w(\la))} Y_w,
  &\text{ if } s_\al w\notin W(\Delta_1,\dots,\Delta_r)\\
\frac{1}{(\al,w(\la))} Y_w+(1+\frac {1}{(\al,w(\la))})
Y_{s_\al w}, &\text{ if } s_\al w \in W(\Delta_1,\dots,\Delta_r),\\
\end{cases}
\end{equation}
for every $\al\in\Pi.$
\end{enumerate}
\end{theorem}

Since the weight $\lambda$ is also the Langlands weight of the
irreducible ladder representation
$\langle\Delta_1,\dots,\Delta_r\rangle,$ we clearly have
\begin{equation}
\langle\Delta_1,\dots,\Delta_r\rangle\cong \mathsf  C(\Delta_1,\dots,\Delta_r),
\end{equation}
for every integral ladder representation.

\section{Ladder representations: functors from category $\C O$ to $\bH$-modules}

In this section, we apply some constructions of Zelevinsky
\cite{Ze-functor} and Arakawa and Suzuki \cite{AS,Su} to the study of
$\bullet$-unitary representations.

\subsection{Category $O$} Let $\fg$ be a complex
reductive Lie algebra with universal enveloping algebra $U(\fg)$. Fix
a Cartan subalgebra $\fk h\subset\fg$, and a Borel subalgebra $\fk
b=\fk h\oplus \fk n.$ Let $R\subset \fk h^*$ denote the roots of $\fg$ with respect
to $\fk h,$ and let $R^+$ be the positive roots with respect to $\fk
b.$ Let $W=N_G(\fk h)/Z_G(\fk h)$ be the Weyl group with length
function $\ell.$ 

Let $\fk n^-$ be the nilradical of the opposite Borel
subalgebra. Let $\Pi$ be the simple roots defined by $R^+$, and for
every root $\al$, let $\al^\vee\in\fk h$ be the coroot. Let $\al_i$,
$i=1,|\Pi|$ denote the simple roots, and $\om_i^\vee$ the corresponding
fundamental coweights. Set $\rho=\frac 12 \sum_{\al\in R^+} \al.$ We
denote by $\langle ~,~\rangle$ the pairing between $\fk h^*$ and $\fk h.$

Define
\begin{align*}
\Lambda&=\{\lambda\in \fk h^*: \langle \lambda,\al^\vee\rangle\in\bZ,~\text{for
  all }\al\in R\};\\
\Lambda^+&=\{\lambda\in \fk h^*: \langle \lambda,\al^\vee\rangle\in\bZ_{\ge 0},~\text{for
  all }\al\in R^+\}.\\
\end{align*}
Let
$O$ denote the category of finitely generated $U(\fg)$-modules, which
are $\fk n$-locally finite and $\fk h$-semisimple. If $X$ is a module
in $O$, let $\Omega(X)$ denote the set of $\fk h$-weights of $X$.

For every $\mu\in \fk h^*$, let $M(\mu)=U(\fg)\otimes_{U(\fk
  b)} \bC v_\mu$ denote the Verma module with highest weight
$\mu$ {and infinitesimal character $\mu+\rho$}.  
Then $M(\mu)\in O$ has a unique simple
quotient, the highest weight module $L(\mu).$ As it is well known,
$L(\mu)$ is a simple finite dimensional module if and only if
$\mu\in\Lambda^+.$

For every $w\in W,$ $\lambda\in\fk h^*$, define
$$w\circ\lambda=w(\lambda+\rho)-\rho.$$

For $X\in O,$ let 
\begin{equation}
H_0(\fk n^-,X)=X/\fk n^- X
\end{equation}
denote the $0$-th $\fk n^-$-homology space, viewed as an $\fk
h$-module. For every $\lambda\in \Lambda^+$ and every finite
dimensional $\fg$-module $\C V$, define the functor
\begin{equation}\label{functor-vs}
F_{\lambda,\C V}: O\to \text{Vect},\quad {F_{\lambda,\C V}}(X)=H_0(\fk
n^-,X\otimes\C V)_\lambda,
\end{equation}
where the subscript stands for the $\lambda$-weight space, and
$\text{Vect}$ denotes the category of $\bC$-vector spaces.

\begin{remark}
In general, $F_{\lambda,\C V}$ need not be an exact functor. However,
if one assumes that $\lambda+\rho\in \Lambda^+$, then $F_{\lambda,\C
  V}$ is exact. See for example \cite[Proposition 1.4.2]{AS}.
\end{remark}

Let $L(\mu)$ be a simple finite dimensional module. Recall that there
exists a resolution of $L(\mu)$ in $O$ defined by
Bernstein-Gelfand-Gelfand:
\begin{equation}\label{BGG}
0\to C_N\to C_{N-1}\to\dots\to C_1\to C_0\to L(\mu)\to 0,\quad\text{where }
C_i=\bigoplus_{w\in W,\ell(w)=i} M_{w\circ \mu}.
\end{equation}
In particular, applying the Euler-Poincar\'e principle, the
identity 
\begin{equation}\label{char-O}
L(\mu)=\sum_{w\in W}\sgn(w) M_{w\circ\mu}
\end{equation}
holds in the Grothendieck group of $O$.
{\begin{proposition}[{\cite[Proposition 1]{Ze-functor}}] 
Fix
  $\lambda\in\Lambda^+$, $\mu\in \Lambda^+$, $\chi\in\Lambda$, and a
  finite dimensional representation $\C V$. 
\begin{enumerate}
\item The functor $F_{\lambda,\C V}$ transforms the BGG resolution
  (\ref{BGG}) into an exact sequence.
\item There are natural $\bC$-linear isomorphisms 
\begin{equation}\label{image-vs}
F_{\lambda,\C V}(M(\chi))=\C V_{\lambda-\chi}\text{ and }
F_{\lambda,\C V}(L(\mu))=\C V_{\lambda-\mu}[\mu],
\end{equation}
{where $\C V_{\lambda-\chi}$ denotes the $(\lambda-\chi)-$weight 
space of $\C V$, and }
$$\C V_{\lambda-\mu}[\mu]=\{v\in \C V_{\lambda-\mu}:
e_\al^{\langle \mu+\rho,\al^\vee\rangle} v=0,\text{ for all
}\al\in\Pi\}.$$ Here $e_\al\in \fk n$ denotes a fixed root vector for
$\al\in\Pi.$ 
\end{enumerate}
\end{proposition}
}

As a corollary, one can transfer formula (\ref{char-O}) via
$F_\lambda.$ This is particularly interesting when the images of
modules in $O$ under $F_\lambda$ admit actions by a different group
(such as in the classical Schur-Weyl duality) or other algebras.

\subsection{The Arakawa-Suzuki functor}
We specialize to $\fg=gl(n,\bC).$ Let $E_{i,j}$ denote the matrix with
$1$ in the $(i,j)$-position and $0$ elsewhere. Let $s_{ij}\in S_n$
denote the transposition. Fix a positive integer $\ell$, and set
\begin{equation}
\C V_\ell=(\bC^n)^{\otimes \ell},
\end{equation}
with the diagonal $\fg$-action.

\begin{remark}
If $\ell=n$, the finite dimensional $\fg$-module $\C V_n$ has the
property that its $0$-weight space is naturally isomorphic to the
standard representation of  $S_n$.
\end{remark}

For every $0\le i<j\le \ell$, consider the operator
\begin{equation}
\Omega_{i,j}=\sum_{1\le k,m\le n} (E_{k,m})_i\otimes (E_{m,k})_j\in
\End(X\otimes\C V_\ell),
\end{equation}
where $(~)_i$ means that the corresponding element acts on the $i$-th
factor of the tensor product. It is well known that $\Omega_{i,j}$,
$1\le i<j\le n$
flips the $i,j$ factors of tensor product, i.e., $\Omega_{i,j}(x\otimes v_1\otimes\dots\otimes
v_i\otimes\dots\otimes v_j\otimes\dots\otimes v_\ell)=x\otimes v_1\otimes\dots\otimes
v_j\otimes\dots\otimes v_i\otimes\dots\otimes v_\ell.$

\begin{lemma}[{\cite[Theorem 2.2.2]{AS},\cite[Lemma 3.1.1]{Su}}] For
  every $X\in O$, the assignment 
\begin{align*}
s_{i,i+1}&\mapsto -\Omega_{i,i+1},\quad &1\le i\le \ell-1,\\
\ep_j^\vee&\mapsto \sum_{0\le i<j}\Omega_{i,j}+\frac{n-1}2,\quad &1\le
j\le \ell,
\end{align*}
extends to an action of the graded Hecke algebra $\bH_\ell$ of
$gl(\ell)$ on $X\otimes\C V_\ell.$
\end{lemma}

Notice the presence of the minus sign in the action of $s_{i,i+1}$
which is not the convention in \cite{AS}. We make this adjustement so
that the results fit with the previous sections. {This is because the
standard modules for $\bH_\ell$ are induced from Steinberg modules to
conform with the Langlands classification, not
the trivial modules as in \cite{AS}.
}

In this way, the functor $F_{\lambda,\C V_\ell}$ from
(\ref{functor-vs}) maps to $\bH_\ell$-modules. Since we will consider
$\lambda$ such that $\lambda+\rho\in \Lambda^+$, this will be an exact
functor. 

In \cite{AS} and
\cite{Su}, the images of Verma modules and highest weight modules are
computed. We recall their results now. 

\smallskip

Let $P(\C V_\ell)\subset \fk h^*$ denote the set of weights {of
  $\C V_\ell$}. If we
identify $\fk h^*$ with $\bC^n,$ then these weights are of the form
$(\ell_1,\dots,\ell_n)$ where $\sum\ell_i=\ell$ {and $\ell_i\ge 0$}.

Assume that $\lambda+\rho\in \Lambda^+$ and let $\mu\in\fk h^*$ be
such that $\lambda-\mu\in P(\C V_\ell).$ Define the multisegment
\cite[(2.2.7)]{Su} 
\begin{equation}
\Phi_{\lambda,\mu}=(\Delta_1,\dots,\Delta_n),\quad
\Delta_i=[\langle\mu+\rho,\ep_i^\vee\rangle,\langle\lambda+
\rho,\ep_i^\vee\rangle-1],  
\end{equation}
and the standard $\bH_\ell$-module
\begin{equation}
\C M(\lambda,\mu)=\langle
\Delta_1,\dots,\Delta_n\rangle=\bH_\ell\otimes_{\bH_{\ell_1}\times\dots\times\bH_{\ell_n}}
(\St\otimes\bC_{\Delta_1})\otimes\dots\otimes(\St\otimes\bC_{\Delta_n}). 
\end{equation}
Let $\C L(\lambda,\mu)$ denote the unique simple quotient of $\C
M(\lambda,\mu).$ The lowest $S_\ell$-type of $\C L(\lambda,\mu)$ is
parameterized by the partition of $\ell$ obtained by ordering 
$\lambda-\mu=(\ell_1,\dots,\ell_n)$ {in decreasing order}.

\begin{theorem}[{\cite[Theorems 3.2.1, 3.2.2]{Su}}]\label{t:AS} Assume
  $\lambda+\rho\in\Lambda^+$ and \newline $\mu\in \lambda-P(\C V_\ell)$.
\begin{enumerate}
\item $F_{\lambda,\C V_\ell}(M(\mu))=\C M(\lambda,\mu)$ as $\bH_\ell$-modules.
\item  If $\mu$ satisfies the condition
\begin{equation}\label{mu-cond}
\langle\mu+\rho,\al^\vee\rangle\in\bZ_{\le 0}\text{ for all }\al\in
R^+\text{ satisfying }\langle\lambda+\rho,\al^\vee\rangle=0,
\end{equation}
then $$F_{\lambda,\C V_\ell}(L(\mu))=\C L(\lambda,\mu).$$
\item If $\mu$ does not satisfy condition (\ref{mu-cond}), then
  $$F_{\lambda,\C V_\ell}(L(\mu))=0.$$
\end{enumerate}
\end{theorem}

Notice that if $\lambda$ in the theorem is such that
$\langle\lambda+\rho,\al^\vee\rangle\ge 1$ for all simple roots $\al,$ then
condition (\ref{mu-cond}) is vacuously true.

\subsection{} We apply the previous results to the ladder
representations. Consider segments $\Delta_i=[a_i,b_i]$, $i=1,n$, such
that $a_1>a_2>\dots$ and $b_1>b_2>\dots.$ Let $\mathsf
C(\Delta_1,\dots,\Delta_n)$ denote the ladder representation for
$\bH_\ell.$  Identify
$(a_1,a_2,\dots,a_n)$ and $(b_1,b_2,\dots,b_n)$ with elements of $\fk
h^*\cong \bC^n$. Set
\begin{equation}
\mu=(a_1,\dots,a_n)-\rho,\quad \lambda=(b_1+1,\dots,b_n+1)-\rho.
\end{equation}
{In coordinates $\rho=(\frac {n-1}2,\frac{n-3}2,\dots,-\frac{n-1}2)$.}

\smallskip

Assume from now on that $(a_1,\dots,a_n)\equiv \rho$ mod $\bZ$. Then
$\lambda$ and $\mu$ just defined satisfy the conditions of Theorem
\ref{t:AS} and, in fact, $\langle \lambda+\rho,\al^\vee\rangle\ge 1$
for all simple roots $\al.$

\begin{proposition}\label{p:char-formula}With the notation as above, we have:
\begin{enumerate}
\item $F_{\lambda,\C V_\ell}(L(\mu))=\mathsf C(\Delta_1,\dots,\Delta_n)$;
\item $\mathsf C(\Delta_1,\dots,\Delta_n)=\sum_{w\in S_n} \sgn(w) \langle
  w\cdot \Delta_1,\dots,w\cdot\Delta_n\rangle,$ where
  $w\cdot\Delta_i:=[a_{w(i)},b_i]$, and the standard representation $\langle
  w\cdot \Delta_1,\dots,w\cdot\Delta_n\rangle$ is understood to be $0$
  if $a_{w(i)}>b_i$ for some $i$.
\end{enumerate}
\end{proposition}

\begin{proof}
Part (1) follows immediately from Theorem \ref{t:AS}(2). For (2), we first
apply the functor $F_{\lambda,\C V_\ell}$ to the BGG formula
(\ref{char-O}), and then identify the images of the Verma modules as
in Theorem \ref{t:AS}(1).
\end{proof}

\begin{remark} Proposition \ref{p:char-formula}(2) recovers the known
  ``determinantal'' character formula for ladder representations of Tadi\'
  c \cite{Ta}, and Lapid-Minguez   \cite[Theorem 1]{LM}, see also \cite{CR}. This approach
  also provides a resolution of the ladder representations which is
  the image of the BGG resolution under the functor. 
\end{remark}

\subsection{} The functor $F_{\lambda,\C V_\ell}$ behaves well with
respect to invariant hermitian (or symmetric bilinear) forms, and in
fact, this is an ingredient in the proof of Theorem \ref{t:AS}(2).  We
recall the results in the 
setting of hermitian rather than symmetric forms, with the obvious
modifications. 

\bigskip
{Recall that $\fk g=gl(n,\bb C)$ viewed as a Lie algebra admits a
complex conjugate linear anti-automorphism $*:A\mapsto \ovl{A^T}.$
A module $X\in O$ is called hermitian if it admits an invariant form
$(~,~)_X$ satisfying
\begin{equation}\label{g-inv}
(Ax,y)_X=(x,A^*y)_X,\quad \text{ for all } A\in\fg=gl(n).
\end{equation}
The standard representation $\bC^n$ is hermitian,  the usual inner
product \newline$$(x,y)_{\bC^n}=\sum_i x_i\overline y_i$$ has property
(\ref{g-inv}). 
}

\bigskip
If $X$ admits an invariant hermitian form,
then $X\otimes \C V_\ell= X\otimes (\bC^n)^{\otimes \ell}$ can be
endowed with the product form.
The following lemma is straightforward.

\begin{lemma}[{\cite[Lemma 4.1.4]{Su}}]\label{inv-forms} Suppose $X$ admits a
  $\fg$-invariant form as in (\ref{g-inv}). Then the form on $X\otimes
  \C V_\ell$ is 
  $\bH_\ell$-invariant with respect to the $\bullet$ star operation of
  $\bH_\ell$.

 If the form on $X$ is nondegenerate (positive definite), then the
 form obtained on $F_{\lambda,\C V_\ell}(X)$ is nondegenerate
 (positive definite).
\end{lemma}

Combining Lemma \ref{inv-forms} with Proposition \ref{p:char-formula},
we obtain as a consequence the known semisimplicity result for ladder
representations \cite{LM}.

\begin{proposition}\label{p:semi}
Every ladder representation $\mathsf C(\Delta_1,\dots,\Delta_n)$ is 
$\bullet-$unitary, and therefore $\bH_M$-semisimple for every
parabolic Hecke subalgebra $\bH_M.$
\end{proposition}

\begin{proof}
Apply Lemma \ref{inv-forms} with $X=L(\mu)$, where $\mu$ and $\lambda$
are as in Proposition \ref{p:char-formula}(1).
\end{proof}

{\begin{remark}
The $\bH_M$-semisimplicity of ladder representations from Proposition \ref{p:semi} is the Hecke algebra equivalent of the semisimplicity of the Jacquet modules of ladder representations proved in \cite{LM}.
\end{remark}
}

\section{Ladder representations: pairs of commuting nilpotent elements}

We relate the $\bA$-semisimple $\bH$-modules to the geometry of pairs
of commuting nilpotent elements considered by \cite{Gi}. Let $\fg$ be
a complex semisimple Lie algebra and $G=\Ad(\fg)$.

\begin{definition}[\cite{Gi}, \cite{EP}]\label{d:nilpair} A pair $\ue=(e_1,e_2)\in \fg\times\fg$ is called a nilpotent pair if
  $[e_1,e_2]=0$ and for all $(t_1,t_2)\in \bC^\times\times\bC^\times$,
  there exists $g\in G$ such that $\Ad(g)(e_1,e_2)=(t_1 e_1,t_2
  e_2)$. In addition:
\begin{enumerate}
\item $\ue$ is called principal if $\dim\fk z_{\fg}(\ue)=\text{rank}~\fg$;
\item $\ue$ is called distinguished if 
\begin{enumerate}
\item $\fk z_\fg(\ue)$ contains no semisimple elements, and
\item there exists a semisimple pair $\uh=(h_1,h_2)\in \fg\times\fg$
  such that $\ad(h_1),\ad(h_2)$ have rational eigenvalues,
$$[h_1,h_2]=0,\ [h_i,e_j]=\delta_{ij} e_j,$$
and $\fk z_\fg(\uh)$ is a Cartan subalgebra.
\end{enumerate}
\item $\ue$ is called rectangular if $e_1,e_2$ can be embedded in
  commuting $sl(2)$ triples.
\end{enumerate}
\end{definition}

By \cite[Theorem 1.2]{Gi}, every principal nilpotent pair $\ue$ is
distinguished, and in fact, the associated semisimple pair $\uh$ has
the property that the eigenvalues of $\ad h_i$ are integral. 

\subsection{} We summarize some of the results from \cite{Gi}.

\begin{theorem}[{\cite[Theorem 3.7, Theorem 3.9, Corollary 3.6]{Gi}}]
\ 

\begin{enumerate}
\item Any two principal nilpotent pairs $\ue$ and $\ue'$ with the same
  associated semisimple pair $\uh$ are conjugate to each other by the
  maximal torus $T=Z_G(\uh).$
\item There are finitely many adjoint $G$-orbits of principal
  nilpotent pairs.
\item For every principal nilpotent pair $\ue$, the centralizer
  $Z_G(\ue)$ is connected.
\end{enumerate}
\end{theorem}

The construction of principal pairs is as follows. Let $\fk p=\fk
l\oplus\fk u$ be a parabolic subalgebra and $e_1\in\fk l$ a principal
nilpotent element. Assume that $e_2\in\fk z_{\fk u}(e_1)$ is a
Richardson element for $\fk p.$ Set $\ue=(e_1,e_2)$. The following are
equivalent:
\begin{enumerate}
\item $\ue$ is a principal nilpotent pair.
\item The orbit $\Ad Z_P(e_1)\cdot e_2$ is Zariski open dense in $\fk
  z_{\fk u}(e_1)$.
\end{enumerate}
{Every principal nilpotent pair is of this form. More precisely, for a
given principal pair $\underline{e}$, let
$\uh=(h_1,h_2)$ be the associated semisimple pair. 
Let $\fg=\oplus_{p,q}\fg_{p,q}$ be the
bigradation of $\fg$ defined by the adjoint action of $\uh.$ Define
\begin{equation}
\fg_{p,*}=\oplus_q\fg_{p,q}\quad\text{ and}\quad \fg_{*,q}=\oplus_p\fg_{p,q},
\end{equation}
and the parabolic subalgebras
\begin{equation}
\fk p^{\text{east}}=\oplus_{p\ge 0}\fg_{p,*}\quad\text{and}\quad \fk
p^{\text{south}}=\oplus_{q\ge 0}\fg_{*,q}
\end{equation}
with Levi subalgebras $\fg^1=\fg_{*,0}$ and $\fg^2=\fg_{0,*}$,
respectively. Then $(e_1,e_2)$ are given by the above construction for $\fk p=\fk
p^{\text{south}}$ and $\fk l=\fg^1.$
}

\subsection{} The notation is motivated by the example $\fg=sl(n).$
In this case let $\sigma$ be a Young diagram visualized as in the
following example: 
$$ \yng(4,2,1)$$

Enumerate the boxes $1,2,\dots,n$ in some order and label the basis of
$\bC^n$ by the box with the corresponding number. Let $e_1,e_2\in
\End(\bC^n)$ be defined as follows: 

$e_1$: sends a basis vector corresponding to a box to the vector
corresponding to the next box on the row (to the east) or $0$ if it's
the last row box;

$e_2$: same as $e_1$ except the direction is down (south) on the
columns.

\begin{theorem}[\cite{Gi}] Suppose $\fg=sl(n)$. Every adjoint
  $G$-orbit of principal nilpotent pairs has a representative obtained
  from a Young diagram by the procedure described above. 

\end{theorem}

The classification of the larger class of distinguished nilpotent
pairs has a similar flavor. Consider $\sigma$ to be a skew Young
diagram, \ie the set difference of two Young diagrams as before with
the same corner. Moreover, assume that $\sigma$ is connected. Define
$\ue=(e_1,e_2)$ as in the Young diagram case, but for the skew diagram
$\sigma.$

\begin{theorem}[{\cite[Theorem 5.6]{Gi}}]\label{t:Gi-dist} The adjoint
  $G$-orbits of 
  distinguished nilpotent pairs are in one to one correspondence, via
  the construction above, with connected skew diagrams $\sigma.$

The rectangular distinguished nilpotent pairs (in the sense of
Definition \ref{d:nilpair}(3)) correspond to rectangular
Young diagrams, {\ie  usual Young diagrams in the shape of rectangles}
(Example \ref{ex:speh}). 
\end{theorem}

\subsection{} We make the connection with ladder
representations. Given an $a\in\bb Z$ and $\sig$ a connected skew
diagram, we associate an integral ladder representation $C(\sig,a)$ as follows.

{Form a skew Young tableau as follows:
the leftmost box of the first row of $\sigma$ gets content
  (the number in the box) $a$,  then the contents increase to
the right and decrease to the left on rows, and stay constant on the
columns. In the following example, $a=2$:
\begin{equation}
\sigma=\young(::\hfil\hfil\hfil,:\hfil\hfil\hfil,\hfil\hfil)\longrightarrow\young(::234,:123,01).
\end{equation}

\newcommand{\ytwo}{{-2}}
\newcommand{\yone}{{-1}}

Suppose
$a_i'$ is the leftmost content in row $i$, while $b_i'$
is the rightmost content. Define the segments:
\begin{equation}\label{Delta-sigma}
\Delta_i=[a_i,b_i],\text{ where }a_i=-(i-1)+a_i'\text{ and } b_i=-(i-1)+b_i'.
\end{equation}
In other words, move the $i$-th row $(i-1)$-units to the left, for
every $i$. In our example,
\begin{equation}
(\Delta_1,\Delta_2,\Delta_3)=([2,4],[0,2],[-2,-1])=\young(::::234,::012,\ytwo\yone).
\end{equation}
\begin{definition}
  \label{d:csiga}
The integral integral ladder representation defiend above will be called
\begin{equation}\label{e:C-sigma}
\mathsf C(\sigma,a):=\langle\Delta_1,\dots,\Delta_r\rangle.
\end{equation}
\end{definition}

}
\

Consider the variety  $\C
B(\ue,\uh)$ of Borel subalgebras of $\fg$ containing the elements 
$(e_1,e_2,h_1,h_2)$. When $\ue$ is distinguished, $\C B(\ue,\uh)$  is 
$0$-dimensional. More precisely, suppose $\fk b\in \C
B(\ue,\uh)$. Since $h_1,h_2\in\fk b,$ also $\fk z_{\fk
  g}(\uh)\subset\fk b.$ As $\ue$ is distinguished, $\fh:=\fk z_{\fk
  g}(\uh)$ is a Cartan subalgebra. This means that every $\fk b\in \C
B(\ue,\uh)$ contains the Cartan subalgebra $\fk h$. Let $W$ be the
Weyl group of $\fh$ in $\fg.$ If $\fk b_0$ is a Borel subalgebra
containg $\fh$ such that $e_1,e_2\in \fk b_0$, then
\begin{equation}\label{e:buv}
\begin{aligned}
\C B(\ue,\uh)=\{w\fk b_0: w\in W(\ue,\fk b_0)\}, \text{ where
}\\W(\ue,\fk b_0)=\{w\in W:~ w^{-1} e_1\in\fk b_0,~w^{-1}
e_2\in\fk b_0\}.
\end{aligned}
\end{equation}
Clearly, if $\fk b_0'=u\fk b_0$ is another Borel subalgebra in $\C
B(\ue,\uh)$, with $u\in W$, then $W(\ue,\fk b_0')=W(\ue,\fk
b_0)u^{-1}.$

\begin{proposition}\label{p:buv} Suppose $\fg=sl(n,\bC)$. Let $\sigma$
  be a connected skew diagram . Let $\ue$ be a distinguished
  nilpotent pair with associated semisimple pair $\uh$, such that
  $\ue$ is attached to $\sigma$ by Theorem
  \ref{t:Gi-dist}.   Let $(\Delta_1,\dots,\Delta_r)$ be the
  multisegment constructed from $\sigma$ by procedure
  (\ref{Delta-sigma}). Then, for every Borel subalgebra $\fk b_0\in \C
  B(\ue,\uh),$ we have
\begin{equation}
W(\ue,\fk b_0)=W(\Delta_1,\dots,\Delta_r) u^{-1},\text{ for some }u\in W,
\end{equation}
where $W(\ue)$, $W(\Delta_1,\dots,\Delta_r)$ are defined in (\ref{e:buv}) and
(\ref{e:W-sigma}), respectively.
\end{proposition}

\begin{proof}
It is sufficient to prove that if $\fk b_0$ is the lower
triangular Borel subalgebra and $\fh$ is the diagonal Cartan
subalgebra, then $W(\ue,\fk b_0)=W(\Delta_1,\dots,\Delta_r).$
If we assign to the boxes of $\sigma$ the standard basis elements of
$\bC^n$ in row order, \eg the boxes of the first row correspond to
$x_1,x_2,\dots, x_{m_1}$, where $m_1=|\Delta_1|$, \etc, then the
nilpotent element $e_1\in\fk b_0$ is a sum
\begin{equation}
e_1=\sum_{i=1}^r X_i,\quad\text{where }
X_1=E_{21}+E_{32}+\dots+E_{m_1,m_1-1},\quad \text{\etc .} 
\end{equation}
Since $w\cdot E_{ij}=E_{w(i),w(j)},$ for $w\in S_n$, it is clear that
the condition $w^{-1}\cdot  e_1\in \fk b_0$ translates to the same
rule as the ``row rule'' (1) used in defining
$W(\Delta_1,\dots,\Delta_r).$

Similarly, $e_2$ is defined using the columns of $\sigma$. Then the
restrictions imposed by the condition $w^{-1}\cdot e_2\in \fk b_0$ are
the same as the ``$45^\circ$-diagonal rule'' (2) used in the
definition of $W(\Delta_1,\dots,\Delta_r).$ Recall that
$(\Delta_1,\dots,\Delta_r)$ is obtained from $\sigma$ by shifting each
row to the left and therefore the column relations become diagonal relations.
\end{proof}

Proposition \ref{p:buv}  has the following immediate corollary.

\begin{corollary}\label{c:buv} Keep the notation from Proposition
  \ref{p:buv}. For every $a\in\bZ$, the $\bA$-weights of the ladder
  representation {$\mathsf C(\sigma,a)$ defined in (\ref{e:C-sigma})} are in one-to-one
correspondence with the points of the variety $\C B(\ue,\uh)$.
\end{corollary}

\ifx\undefined\bysame
\newcommand{\bysame}{\leavevmode\hbox to3em{\hrulefill}\,}
\fi

\end{document}